\documentclass[oneside]{tac}

\title{Modeling Stable One-Types}
\author{Niles Johnson, Ang\'elica M. Osorno}
\copyrightyear{2012}
\address{%
N.J.\ 
Department of Mathematics,
The Ohio State University, Newark\\
Newark, OH 43055
USA\\
{\tt johnson.5320@osu.edu}\\[1pc]
A.O.\ 
Department of Mathematics,
University of Chicago\\
Chicago, IL 60637
USA\\
{\tt aosorno@math.uchicago.edu}\\
}

\keywords{stable homotopy one-type, Picard groupoid}
\amsclass{18B40, 18D10, 55P42, 55S45}
\date{19 July, 2012}

\usepackage[all,cmtip]{xy}
\xyoption{curve}      

\usepackage{amsmath,amssymb}

\usepackage{etoolbox}
\providebool{extrarefs}
\setbool{extrarefs}{true}  

\providebool{bibsortlabels}
\setbool{bibsortlabels}{true} 

\DeclareFontFamily{OMS}{rsfs}{\skewchar\font'60}
\DeclareFontShape{OMS}{rsfs}{m}{n}{<-5>rsfs5 <5-7>rsfs7 <7->rsfs10 }{}
\DeclareSymbolFont{rsfs}{OMS}{rsfs}{m}{n}
\DeclareSymbolFontAlphabet{\scr}{rsfs}
\usepackage{mathrsfs}

\newcommand{\un}[1]{\underline{#1}}

\newcommand{\sC}{\scr{C}}
\newcommand{\sD}{\scr{D}}
\newcommand{\sE}{\scr{E}}
\newcommand{\sF}{\scr{F}}
\newcommand{\sG}{\scr{G}}
\newcommand{\sH}{\scr{H}}
\newcommand{\sO}{\scr{O}}
\newcommand{\sS}{\scr{S}}
\newcommand{\sT}{\scr{T}}

\newcommand{\bN}{\mathbb{N}}
\newcommand{\bS}{\mathbb{S}}
\newcommand{\bZ}{\mathbb{Z}}

\newcommand{\al}{\alpha}
\newcommand{\epz}{\varepsilon}
\newcommand{\De}{\Delta}

\newcommand{\fto}{\xrightarrow}
\newcommand{\fot}{\xleftarrow}
\newcommand{\cn}{\colon}

\DeclareMathOperator{\Ker}{\mathrm{Ker}}
\DeclareMathOperator{\Coker}{\mathrm{Coker}}
\DeclareMathOperator{\CCoker}{\mathbb{C}\mathrm{oker}}

\DeclareFontFamily{OT1}{pzc}{}
\DeclareFontShape{OT1}{pzc}{m}{it}{<-> s * [1.140] pzcmi7t}{}  
\DeclareMathAlphabet{\mathpzc}{OT1}{pzc}{m}{it}

\newcommand{\gpd}{\mathpzc{Gpd}}
\newcommand{\Top}{\mathpzc{Top}}
\newcommand{\cat}{\mathpzc{Cat}}
\newcommand{\sympic}{\mathpzc{Pic}}
\newcommand{\stab}{\sS ^1_0}
\newcommand{\Ho}[1]{\mathpzc{Ho}(#1)}

\newcommand{\xym}{\xymatrix}

\newcommand{\hty}{\simeq}
\newcommand{\iso}{\cong}

\usepackage[T1]{fontenc}

\usepackage[short,nodayofweek]{datetime}

\usepackage{xcolor}

\usepackage[unicode=true, pdfusetitle,
 bookmarks=true,bookmarksnumbered=false,
 breaklinks=false,
 backref=false,
 colorlinks=true,
 linkcolor=blue!70!black,
 citecolor=black,
 urlcolor=blue!78!red,
 final
]{hyperref}
\hypersetup{pdfkeywords={stable one-type, symmetric monoidal bicategory}}

\usepackage{bookmark} 

\usepackage[capitalise]{cleveref} 

\crefname{lem}{Lemma}{Lemmas}
\crefname{thm}{Theorem}{Theorems}
\crefname{defn}{Definition}{Definitions}
\crefname{prop}{Proposition}{Propositions}
\crefname{rmk}{Remark}{Remarks}
\crefname{cor}{Corollary}{Corollaries}

\renewcommand{\theenumi}{\textit{\roman{enumi}}}

\renewcommand{\theenumii}{\textit{\alph{enumii}}}

\theoremstyle{plain}

\newtheorem{thm}{Theorem}

\makeatletter
\def\xnstate#1{\subsection*{\sc#1}%
\begingroup\em\@ifnextchar[{\@credit}{}}

\def\xnnewtheorem#1{\@ifnextchar[{\@xnnewthm{#1}}{\@xnnewthm{#1}[]}}

\def\@xnnewthm#1[#2]#3{\@ifnextchar[{\@@xnnewthm{#1}[#2]{#3}}{\@@xnnewthm{#1}[#2]{#3}[ ]}}

\def\@@xnnewthm#1[#2]#3[#4]{\expandafter\def\csname #1\endcsname{%
\xnstate{#3.}\@ifnextchar[{\@credit}{}}%
 \expandafter\let\csname
     end#1\endcsname\endstate}

\makeatother

\xnnewtheorem{thmA}{Theorem A}
\xnnewtheorem{thmB}{Theorem B}
\xnnewtheorem{thmC}{Theorem C}

\newtheorem{cor}[subsection]{Corollary}

\newtheorem{prop}[subsection]{Proposition}

\newtheorem{lem}[subsection]{Lemma}

\newtheoremrm{defn}[subsection]{Definition}

\newtheoremrm{rmk}[subsection]{Remark}

\numberwithin{equation}{section}

\begin{document}


\maketitle

\begin{abstract}
  Classification of homotopy $n$-types has focused on developing
  algebraic categories which are equivalent to categories of
  $n$-types.  We expand this theory by providing algebraic models of
  homotopy-theoretic constructions for stable one-types.  These
  include a model for the Postnikov one-truncation of the sphere
  spectrum, and for its action on the model of a stable one-type.  We
  show that a bicategorical cokernel introduced by Vitale models the
  cofiber of a map between stable one-types, and apply this to develop
  an algebraic model for the Postnikov data of a stable one-type.
\end{abstract}

\section*{Introduction}

The homotopy category of groupoids is equivalent to the homotopy
category of unstable one-types, via the classifying space and
fundamental groupoid functors. This is one of the well-known results
from a large body of work around the ``algebraic homotopy'' outlined
by J.H.C.~Whitehead in his 1950 address to the International Congress
of Mathematicians.  Crossed modules classify unstable two-types, and
Conduch\'e gives a genaeralization to unstable three-types
\cite{conduche84}.

A related body of work focuses on stable homotopy type.  Stable
one-types are classified by Picard groupoids, i.e., group-like
symmetric monoidal groupoids.  This is a well-known result for which
we give a new proof in Section~\ref {sec:stable-one-types}. Picard groupoids
were first introduced in the thesis of S\'inh Hoang Xuan
\cite{sinh1975}, where the author gives a thorough algebraic
classification theorem.  Since then, various results have further
established the link between Picard groupoids and stable one types.
Garz\'on and Miranda \cite{GM97Homotopy} develop a model structure for
the categories of Picard groupoids, identifying the path and cylinder
constructions therein. They use this setting to model homotopy classes
of maps between spaces with nontrivial homotopy groups in degrees $n$
and $n+1$ for $n \ge 1$. Garz\'on-Miranda-del R\'io
\cite{GMdR02Tensor} give categorical models for the $n^{\text{th}}$
homotopy groupoid of a space for $n \ge 2$, showing that the resulting
monoidal categories are braided for $n = 2$ and symmetric for $n \ge
3$. As a generalization of Eilenberg-Mac Lane cohomology,
Bullejos-Carrasco-Cegarra \cite{BCC93Cohomology} define a cohomology
of simplicial sets with coefficients in a Picard groupoid.  Their work
uses this to give alternate categorical models for spaces with homotopy
groups in degrees $n$ and $n+1$ for $n \ge 3$---the stable range.

Our proof that Picard groupoids classify stable one-types is given in
Theorem~\ref {thm:stable-1-type-symm-picard-equiv} using the perspective of
$E_\infty$ action on the categorical and topological objects.  Our
main results go beyond the basic classification to describe the
homotopical structure of stable one-types through corresponding
structure of Picard groupoids.

Specifically, we study the decomposition of stable one-types by their
Postnikov data.  This consists of abelian groups $\pi_0$ and $\pi_1$,
and a single $k$-invariant, which is a map of Eilenberg-Mac Lane
spectra $H\pi_0 \to \Sigma ^2 H\pi_1$. Using the isomorphism
$[H\pi_0,\Sigma ^2 H\pi_1]\cong Hom(\pi_0 /2\pi_0,\pi_1)$
\cite[(2.7)]{EM54groups}, this $k$-invariant can be identified with
the quadratic map $\eta ^{\ast} : \pi_0 \to \pi _1$, induced by
precomposition with the Hopf map $\eta : S^3 \to S^2$ \cite[\S
  8]{BM2008shtpy}.  Our main results model the Postnikov data and
sphere action on a stable one-type directly in terms of Picard
groupoid data.  Note, in particular, that the target of the Postnikov
invariant is a stable two-type (of a special kind).  Hence our
algebraic models lead naturally toward models for stable two-types.


Our main results are as follows.  For Picard groupoids $\sC$ and $\sD$
and a  symmetric monoidal functor $F : \sC \to \sD$, we describe a symmetric
monoidal bicategory $\Coker(F)$ first introduced by Vitale
\cite{vitale2002pbe}.  In Section~\ref {sec:cokernels} we apply a long exact
sequence argument to prove the following result as
Theorem~\ref {thm:cok-sym-mon} and~Corollaries~\ref {cor:coker-and-hocofib} and~ \ref {cor:Postnikov-invars}.
\begin{thmA}
  Let $F: \sC \to \sD$ be a functor of Picard groupoids.
  Then there is a bigroupoid $\Coker(F)$ and a
  natural pseudofunctor
  \[
  C_F \cn \sD \to \Coker(F)
  \]
  which models the stable cofiber in the following sense:
  \begin{enumerate}
  \item $\Coker(F)$ is symmetric monoidal and $C_F$ is a symmetric monoidal pseudofunctor.
  \item Taking classifying spaces yields a cofibration sequence of grouplike
    $E_\infty$ spaces:
    \[
    B\sC \to B\sD \to B\Coker(F).
    \]
  \item When $\sD = \sC_0$ is the discrete category of isomorphism classes of
    objects in $\sC$ and $F = \al_0$ is the induced monoidal functor,
    we have an equivalence with the Postnikov tower of $B\sC$:
    \[\xym{
      B\sC \ar[r] \ar@{=}[d] 
      & K(\pi_0, 0) \ar[r] \ar[d]^-{\hty} 
      & K(\pi_1, 2) \ar[d]^-{\hty} \\
      B\sC \ar[r]^-{B\al_0}
      & B\sC_0 \ar[r]^-{BC_{\alpha_0}} 
      & B\Coker(\al_0) 
    }\]
  \end{enumerate}
\end{thmA}
Our work uses a strictification result which is somewhat stronger than
one could expect for general symmetric monoidal categories, and it may
be of independent interest.  This is Theorem~\ref {prop:skel-strictification}:
\begin{thmB}
  Every Picard groupoid is equivalent as a symmetric monoidal category
  to one which is both skeletal and permutative.
\end{thmB}

Our approach also reveals that the action of the truncated sphere
spectrum on a stable one-type is present in the algebraic model. This
is hinted at in the unstable literature
\cite{BCC93Cohomology,GMdR02Tensor} but not described explicitly.  We
prove the following as
Propositions~\ref {prop:free-symm-Pic-cat},  \ref {prop:model-trunc-sphere} and~ \ref {prop:model-eta-action}.
\begin{samepage}
  \begin{thmC} There is a Picard groupoid $\bS$ which models the
    one-type of the sphere spectrum in the following sense:
    \begin{enumerate}
    \item The Picard groupoid $\bS$ is the free Picard
      groupoid on one object.
    \item The classifying space $B\bS$ is the Postnikov
      1-truncation of $QS^0$.
    \item Let $\sC$ be a Picard groupoid.  There is a natural action
      of $\bS$ on $\sC$ such that the induced action of $B\bS$ on
      $B\sC$ is equivalent to the action of the truncated
      sphere spectrum on $B\sC$.
    \end{enumerate}
  \end{thmC}
\end{samepage}

This work is a proving ground for a larger project joint with J.P.~May
which models stable two-types via symmetric monoidal bicategories.
The top Postnikov invariant in that case lands in a stable 3-type,
which should be modeled by a symmetric monoidal tricategory; one
purpose of our program is to use this approach as leverage to
understand symmetric monoidal structure on higher weak $n$-categories.

\subsection*{Acknowledgements.} The authors wish to thank Peter May for
helpful conversations and Nick Gurski for suggesting the proof of
Proposition~\ref {prop:free-symm-Pic-cat}.  They are also grateful for
the suggestions of an anonymous referee.

\section{Stable one-types}\label{sec:stable-one-types}

Let $\sS$ denote any symmetric monoidal model category of spectra. We
denote by $\stab$ the full subcategory of $\sS$ whose objects are
spectra with all homotopy groups equal to zero except at levels 0 and
1. The objects of this category are called \emph{stable one-types}. A
map between stable one-types is a stable equivalence if it induces
isomorphisms of homotopy groups.

\begin{defn}
  Let $(\sC, \oplus, I)$ be a symmetric monoidal category. An object
  $x$ is \emph{invertible} if there exists an object $y$ and an
  isomorphism
  \[
  \epz \cn y \oplus x \to I.
  \]
  If such a $y$ exists it is unique up to isomorphism.  When one
  exists, we will sometimes use $x^*$ to denote a specified inverse of
  $x$.
\end{defn}

\begin{defn}
  A \emph{Picard groupoid} $\sC$ is a symmetric monoidal groupoid such
  that every object is invertible.  The isomorphism classes of objects
  form an abelian group denoted $\pi_0 \sC$, and the endomorphisms of
  the identity object $I$ form an abelian group denoted $\pi_1\sC$.
\end{defn}

\begin{defn}
  The category $\sympic$ has as objects the Picard groupoids and as
  morphisms strong symmetric monoidal functors.  For symmetric
  monoidal functors into groupoids, the notions of lax and strong
  monoidal coincide.  Throughout the paper our monoidal functors
  are assumed to be strong monoidal.  A symmetric monoidal functor is
  a \emph{weak equivalence} if it is an equivalence of the underlying
  categories.
\end{defn}

\begin{rmk}
  The term \emph{Picard category} is used in some literature for what
  we call a Picard groupoid.  Although some will interpret our
  terminology as redundant, we hope others will find it maximally
  comprehensible.
\end{rmk}

\ifbool{extrarefs}{
The classification of stable one-types by Picard groupoids appears
explicitly and implicitly in various parts of the literature.  For
example, Patel \cite[\S 5]{patel2008} shows that there is an
equivalence between the homotopy categories of stable one-types and
Picard groupoids, making precise a ``vauge idea'' of Drinfeld \cite[\S
5.5]{Dri06Infinite}.  This is also sketched by Hopkins-Singer in
\cite[\S B]{HS02Quadratic} and Ganter-Kapranov in \cite[\S
3]{GK11Symmetric}.  An equivalent result of Bullejos-Carrasco-Cegarra
appears in \cite[\S5]{BCC93Cohomology}, where the authors prove that
the homotopy category of spaces with nontrivial homotopy groups
$\pi_n$ and $\pi_{n+1}$, $n \ge 3$, is equivalent to the homotopy
category of Picard groupoids.  We give another proof of this result
based on compatibility of the fundamental groupoid and classifying
space functors with $E_\infty$ actions.
}{
The classification of stable one-types by Picard groupoids appears
explicitly and implicitly in various parts of the literature.  For
example, Patel \cite[\S 5]{patel2008} shows that there is an
equivalence between the homotopy categories of stable one-types and
Picard groupoids, and this is also sketched by Hopkins-Singer in
\cite[\S B]{HS02Quadratic}.  An equivalent result of
Bullejos-Carrasco-Cegarra appears in \cite[\S5]{BCC93Cohomology},
where the authors prove that the homotopy category of spaces with
nontrivial homotopy groups $\pi_n$ and $\pi_{n+1}$, $n \ge 3$, is
equivalent to the homotopy category of Picard groupoids.  We give
another proof of this result based on compatibility of the fundamental
groupoid and classifying space functors with $E_\infty$ actions.
}

\begin{thm}\label{thm:stable-1-type-symm-picard-equiv}
  There is an equivalence between the categories $\Ho{\stab}$ and
  $\Ho{\sympic}$ induced by the fundamental groupoid and classifying
  space functors.
\end{thm}
\begin{proof}
  We first recall that the homotopy category of connective spectra is
  equivalent to the homotopy category of group-like $E_{\infty}$
  spaces, and this equivalence descends to the category of stable
  one-types and the subcategory of group-like $E_{\infty}$ spaces with
  no higher homotopy groups. Thus we can work in the context of
  $E_{\infty}$ spaces.

  It is a classical result that the classifying space and fundamental
  groupoid functors give an equivalence
  \[\xym{
    \Pi_1 \cn \Ho{\Top_0^1} \ar@<.5ex>[r]^-{\hty} & \Ho{\gpd} \cn B \ar@<.5ex>[l]
  }\]
  where $\Top_0^1$ is the category of one-type spaces.  Thus it suffices
  to show that $\Pi_1$ and $B$ induce an equivalence between the
  homotopy categories of group-like $E_\infty$ one-types and Picard
  groupoids.

  Let $\sO$ be the categorical Barrat-Eccles operad---its $j$th
  category $\sO(j)$ is the translation groupoid of the action of
  $\Sigma _j$ on itself and its algebras are permutative categories
  \cite{may72geo}.  Then $B\sO$ is an $E_\infty$ operad in $\Top$: If
  $\sC$ is a symmetric monoidal category, then $B\sC$ is an $E_\infty$
  space and if $\sC$ is a Picard groupoid, then $B\sC$ is a group-like
  $E_\infty$ one-type. If $F \cn \sC \to \sD$ is a functor between
  Picard groupoids, then $BF \cn B\sC \to B\sD$ is an $E_\infty$ map.

  The operad $\Pi_1 B \sO$ is an $E_\infty$ operad in categories, and
  $\Pi_1$ preserves products.  If $X$ is an $E_\infty$ one-type, then $\Pi_1
  X$ is a symmetric monoidal groupoid.  Moreover, $\pi_0 X \iso \pi_0
  \Pi_1 X$, so $\Pi_1 X$ is a Picard category if $X$ is group-like. If
  $f\cn X \to Y$ is a map of $E_\infty$ spaces, then $\Pi_1 f$ is a
  symmetric monoidal functor.

  Now consider the equivalence $\sC \to \Pi_1B\sC$. Since this functor
  is part of a natural transformation of functors from $\gpd$ to
  itself, we have functors $\sO(j)\times \sC^j\to \Pi_1B\sO(j)\times
  (\Pi_1B\sC)^j$ that commute with the structure maps of the algebras
  $\sC$ and $\Pi_1B\sC$, thus showing that the equivalence $\sC \to
  \Pi_1B\sC$ is a symmetric monoidal functor. A similar argument shows
  that for a stable one-type $X$, the weak equivalence $X\to B\Pi_1X$ is
  an $E_{\infty}$ map.
\end{proof}

\begin{rmk}
  One can actually show that the fundamental groupoid of an $E_3$
  algebra is symmetric monoidal. Indeed, if $\sO$ is an $E_3$ operad
  in $\Top$, then $\Pi_1\sO$ is an $E_{\infty}$ operad in $\cat$. This
  is because the fundamental groupoid depends only on the homotopy
  one-type of a space, and thus the obstructions to lifting an $E_3$
  structure to an $E_\infty$ structure on a groupoid
  vanish. Alternatively, one can provide an explicit argument using
  specific points of the little 3-cubes operad $\sC_3$ to prove that
  if $X$ is an algebra over $\sC_3$ then $\Pi _1 X$ is a symmetric
  monoidal category. An example of this strategy can be found in
  \cite[Theorem 15]{gurski11}, where the author proves that the
  fundamental 2-groupoid of an algebra over the little 2-cubes operad
  is braided monoidal.
\end{rmk}



\section{Strictification}

In this section we prove a strictification result for skeletal Picard
groupoids.  The result is an algebraic reflection of the fact that the
first $k$-invariant of a connected double loop space is trivial
\cite[Theorem 5.8]{BC1997iterloop}.

\begin{defn}
  A Picard groupoid is \emph{permutative} if it is strictly
  associative and strictly unital.
\end{defn}

\begin{thm}\label{prop:skel-strictification}
  Every Picard groupoid is equivalent to one which is both skeletal
  and permutative.
\end{thm}
\noindent The proof appears after Proposition~\ref {prop:braidings-ab-3-cocycles},
which classifies Picard groupoids by symmetric 3-cocycles.
Analogous results for stable crossed modules appear in \cite{BC1997iterloop}.

\begin{defn}[Symmetric 3-cocycle]
  Let $G$ be an abelian group and $M$ a trivial $G$-module.  A
  \emph{symmetric 3-cocycle} for $G$ with coefficients in $M$ is a pair $(h,c)$
  where $h$ is a normalized 3-cocycle:  for $x, y, z \in G$
  \begin{align*}
    h(x,0,z) & = 0, \\
    h(x,y,z) + h(u, x+y, z) + h(u,x,y) & = h(u, x, y+z) + h(u+x, y, z),
\intertext{and $c \cn G^2 \to M$ is a function satisfying}
    h(y, z, x) + c(x, y+z) + h(x,y,z) &
    = c(x,z) + h(y, x, z) + c(x, y),\\
    c(x,y) & = -c(y,x).
  \end{align*}
  We say two symmetric 3-cocycles $(h,c)$ and $(h',c')$ are cohomologous if there exists a function $k \cn G^2\to M$ satisfying
  \begin{align*}
    k(x,0) &=k(0,y) = 0,\\
    h(x,y,z)-h'(x,y,z) & =k(y,z)-k(x+y,z)+k(x,y+z)-k(x,y),\\
    c(x,y)-c'(x,y) & = k(x,y)-k(y,x).
  \end{align*}
  We denote the group of cohomology classes of symmetric 3-cocycles by
  $H^3_{\text{sym}}(G;M)$.
\end{defn}

\begin{defn}[{\cite[Chapter 2, \S 2]{sinh1975}},{\cite[\S3]{JS1993btc}}]
  Let $G$ be an abelian group, $M$ a trivial $G$ module, and $(h,c)$
  a symmetric 3-cocycle for $G$ with coefficients in $M$.  We define a
  skeletal Picard groupoid $\sT = \sT(G, M, (h,c))$ whose
  objects are the elements of $G$ and whose morphisms are given by
  \[
  \sT(x,y) = 
  \begin{cases}
    M & \text{ if } x = y\\
    \emptyset & \text{ if } x \neq y.
  \end{cases}
  \]
  Composition is defined by the addition in $M$ and the monoidal
  structure is addition in $G$.  The associativity is determined by
  \[
  h(x,y,z) \cn (x+y) + z \to x + (y + z)
  \]
  and the symmetry isomorphism is determined by
  \[
  c(x,y) \cn x + y \to y + x.
  \]
  The axioms of a symmetric 3-cocycle are precisely the axioms for
  compatibility of the symmetry and associativity in a skeletal symmetric
  monoidal groupoid.
\end{defn}

\begin{prop}[{\cite[Chapter II, \S 2.1]{sinh1975},{\cite[\S 3]{JS1993btc}}}]
  \label{prop:braidings-ab-3-cocycles}
  Every Picard groupoid $\sC$ is equivalent to a skeletal one
  of the form $\sT(G,M,(h,c))$ where $G=\pi_0 \sC$, $M=\pi_1 \sC$, and
  $(h,c)\in H^3_{\text{sym}}(G;M)$ is a symmetric 3-cocycle that
  represents the associativity and the symmetry.

  Two skeletal Picard groupoids $\sT(G,M,(h,c))$ and
  $\sT(G,M,(h',c'))$ are equivalent if and only if $(h,c)$ and
  $(h',c')$ are cohomologous.
\end{prop}

\begin{proof}[of Theorem~\ref {prop:skel-strictification}.]
  By Proposition~\ref {prop:braidings-ab-3-cocycles}, it suffices to consider a
  skeletal Picard groupoid $\sC = \sT(G,M,(h,c))$.  Moreover, any
  abelian group $G$ is a filtered colimit of finitely generated
  abelian groups, and any finitely generated abelian group is a direct
  sum of cyclic groups. Thus, by making use of the K\"unneth theorem
  and colimits over finitely generated abelian groups, it suffices to
  consider the case where $G$ is cyclic.  This strategy for studying
  Picard groupoids appears in \cite[26.4]{EM54groups} and
  \cite[3.2]{JS1993btc}.

  Let $G$ be a cyclic group. We now prove that there is some $c'$ such
  that $[(h,c)]=[(0,c')]$ in $H^3_{\text{sym}}(G; M)$.  This is
  immediate in the infinite cyclic case since $H^3(\bZ; M)=0$.  Now
  suppose $G = \bZ/n$.  The third cohomology group is
  \[
  H^3(\bZ/n; M) \iso \{\mu \in M| n\mu = 0\}.
  \]
  Following Joyal and Street \cite[\S 3]{JS1993btc} we have an
  explicit formula for cocycle representatives corresponding to $\mu$:
  \[
  h_{\mu}(x,y,z) =
  \begin{cases}
    0 & \text{for } y+z < n\\
    x \mu & \text{for } y+z \ge n
  \end{cases}
  \]
  where $x, y, z$ are taken to be integers in $\{0 , \ldots, n-1\}$
  and addition is performed over the integers to determine the values
  of $h_\mu$.

  A calculation shows
  \begin{align}
    \label{eq:cocycle-equiv}
    (h,c) & \sim (h_{n c(1,1)}, \rho_{c(1,1)}),
  \end{align}
  where $\rho_{c(1,1)}$ denotes the symmetry given by
  \[
  (x, y) \mapsto xy \cdot c(1,1).
  \]
  This equivalence of cocycles determines a symmetric monoidal
  equivalence of the corresponding symmetric monoidal categories.
  Since the braiding on $\sC$ is a symmetry, we have $c(x,y) =
  -c(y,x)$.

  Now note that if $(h,c)$ is a symmetric 3-cocycle, then $n c(1,1) =
  0$ by Lemma~\ref {lem:symm-3-cocycle}, so $h_{nc(1,1)} = h_0 = 0$.
  Therefore Eq.~\textup {(\ref {eq:cocycle-equiv})} shows that $\sC$ is equivalent (as
  a Picard groupoid) to one whose representing cocycle is
  $(0,\rho_c(1,1))$ and thus is both skeletal and permutative.
\end{proof}

\begin{lem}\label{lem:symm-3-cocycle}
  If $(h,c)$ is a symmetric 3-cocycle of $\bZ/n$ with coefficients in
  $M$, then $n c(1,1) = 0$.
\end{lem}
\begin{proof}
  Since $c$ is symmetric, $c(1,1) = -c(1,1)$, and thus $2 c(1,1) =
  0$.  If $n$ is even, then the result follows; if $n$ is odd, we make
  use of the identity $c(x,x) = x^2 c(1,1)$ for all $x \in \bZ/n$:
  \[
  c(1,1) = (-1)^2 c(1,1) = c(-1,-1) = c(n-1, n-1) = (n-1)^2 c(1,1).
  \]
  For $n$ odd, $n-1$ is even and thus the last term is zero.
\end{proof}

The calculation of Eq.~\textup {(\ref {eq:cocycle-equiv})} shows that the symmetry $c$
completely determines the cohomology class of the symmetric 3-cocycle
$(h, c)$ of a skeletal Picard groupoid.  But \cite[\S 3]{JS1993btc}
shows that the symmetry determines and is determined by the quadratic
map
\[
q = c \circ \De \cn G \to M.
\]
\begin{defn}[Quadratic map]
  A map $q \cn G \to M$ is \emph{quadratic} if
  \begin{align*}
    q(x) & = q(-x), \\
    q(x + y + z) + q(x) + q(y) + q(z) & = q(y+z) + q(z+x) + q(x+y).
  \end{align*}
\end{defn}
Eilenberg and Mac Lane \cite{EM54groups}, and Loday \cite{Loday82} show
that the set of quadratic maps $q: G \to M$ is isomorphic to the set
of homotopy classes of maps $[K(G, n) , K(M, n+2)]$ for $n \ge 3$,
which is the set of stable homotopy classes of maps $[K(G,0),
K(M,2)]_{\text{stable}}$.  This is the set of possible Postnikov
invariants of a stable one-type with $\pi_0 = G$ and $\pi_1 = M$.
Thus we have the following refinement of
Theorem~\ref {thm:stable-1-type-symm-picard-equiv}:

\begin{cor}\label{cor:model-postnikov-invar}
  The stable one-types with $\pi_0 = G$ and $\pi_1 = M$ are classified
  by the symmetric structures on a skeletal and permutative monoidal
  groupoid with objects $G$ and each endomorphism group isomorphic to
  $M$.
\end{cor}

\begin{rmk}
  The contrast between triviality of unstable $k$-invariants and
  non-triviality of stable $k$-invariants may be worth clarifying:
  When modeling connected spaces with nontrivial $\pi_1$ and $\pi_2$,
  it is the associativity of a monoidal groupoid (with invertible
  objects) which gives the first $k$-invariant of the corresponding
  space.  However when modeling spectra with nontrivial $\pi_0$ and
  $\pi_1$ it is the \emph{symmetry} of a Picard groupoid which gives
  the first (stable) $k$-invariant.  A consequence of
  Theorem~\ref {thm:stable-1-type-symm-picard-equiv},
  Corollary~\ref {cor:model-postnikov-invar}, and \cite[Theorem
    5.8]{BC1997iterloop} is that the first $k$-invariant of a stable
  one-type is unstably trivial.
\end{rmk}



\section{The truncated sphere spectrum}\label{sec:truncated-sphere}

We now define a skeletal and permutative Picard groupoid $\bS$ and
explain how it is an algebraic model of the truncated sphere spectrum.
The objects of $\bS$ are the integers under addition, and the
morphisms are given by
\[
\bS(m,n) =
\begin{cases}
  \emptyset & \text{if } m \neq n\\
  \bZ/2 & \text{if } m = n.
\end{cases}
\]
We let $\eta_n$ denote the nontrivial element of $\bS(n,n)$ for each
$n$.  The monoidal structure is symmetric, with the symmetry
isomorphism given by
\[
c_{m,n}=\begin{cases}
   0 & \text{if } mn \text{ is even}\\
   \eta_{m+n} & \text{if } mn \text{ is odd}.
   \end{cases}
\]
Note that this symmetry isomorphism gives rise to the stable quadratic map
$q:\bZ \to \bZ/2$ given by the mod 2 map.

The Picard groupoid $\bS$ is closely related to the category of finite
sets, as we now describe.  Let $\sE$ be the skeletal category whose
objects are the finite sets $\un{0} = \emptyset$, $\un{n}=\{1, 2,
\dots, n\}$ and whose morphism sets are given by the symmetric
groups. This is a permutative category, with sum given by sum in $\bN$
and with symmetry isomorphism $ c_{m,n}^{\oplus}$ given by the
permutation that sends $(1, 2, \dots, m+n)$ to $(m+1, m+2, \dots, m+n,
1, 2, \dots, m)$.  Note that $\sE$ is skeletal and is equivalent to
the category of finite sets.

There is a symmetric monoidal functor
\[
\xi: \sE \to \bS
\]
given on objects by the inclusion of $\bN$ into $\bZ$ and on morphisms
by the sign homomorphism $\Sigma_n \to \bZ/2$.  This is the functor
that first abelianizes the group of endomorphisms of each object
$\un{n}$, and then includes into $\bS$.

The next three results justify our notation for $\bS$ by showing that
it is the free Picard groupoid on one object, its
classifying space is the Postnikov 1-truncation of $QS^0$, and its
natural action on a Picard groupoid $\sC$ is a model for the
action of the truncated sphere spectrum on $B\sC$.

\begin{prop}\label{prop:free-symm-Pic-cat}
  The Picard groupoid $\bS$ is symmetric monoidally equivalent
  to the free Picard groupoid on one object,
  $\sF_{Pic}(\ast)$.
\end{prop}
\begin{proof}
  The free Picard groupoid functor $\sF_{Pic}$ is equal to the
  composite of the free symmetric monoidal groupoid functor,
  $\sF_{symMon}$, with the functor that freely adjoins inverses for
  objects, $\sF_{inv}$. The free symmetric monoidal category on one
  object, $\sF_{symMon}(*)$, is symmetric monoidally equivalent to the
  category $\sE$ defined above.

  We now show that $\bS$ satisfies the universal property for
  $\sF_{inv}(\sE)$: Let $\sC$ be a Picard category and $G:\sE\to \sC$
  a symmetric monoidal functor.  We construct a symmetric monoidal
  functor $H$ making the diagram commute:
  \[
  \xym{
    \sE \ar[r]^{G} \ar[d]_{\xi} & \sC.\\
    \bS \ar@{-->}[ru]_{H} &
  }
  \]
  For every object $x\in \sC$ fix an inverse $x^{\ast}$. We define $H$
  on objects as
  \[
  H(n)=\begin{cases}
    G(n) & \text{if } n\geq 0\\
    G(|n|)^{\ast} & \text{if } n<0.
  \end{cases}
  \]
  To define $H$ on morphisms, note that $\sC(x,x)$ is an abelian group
  for all $x\in \sC$ and therefore $G$ factors through the
  abelianization of $\sE(n,n)$ and hence through $\xi$.  This
  factorization determines $H$ on the endomorphism group of $n$ for $n
  \ge 0$, and the values of $H$ on endomorphisms of negative $n$ are
  determined by translation. It is easy to see that $H$ is a symmetric
  monoidal functor.
  
  Now let $H'$ be another symmetric monoidal functor making the
  diagram commute. Note that for $n\geq 0$, we must have
  $H(n)=G(n)=H'(n)$.  On the other hand we have natural isomorphisms
  \[
  H(n) \oplus H(-n) \fto{\iso} I_\sC \fot{\iso} H'(n) \oplus H'(-n)
  \]
  and hence natural isomorphisms
  \[
  H(-n) \fto{\iso} H(n)^* = H'(n)^* \fot{\iso} H'(-n).
  \]
  These assemble to form a monoidal natural
  isomorphism between $H$ and $H'$.
\end{proof}

\begin{rmk}\label{rmk:S-biperm}
  Although a model for one-types of ring spectra is beyond the scope
  of this paper, we do note that $\bS$ has a second symmetric monoidal
  structure, so that it is a bipermutative groupoid. This second
  monoidal structure is given by:
  \begin{align*}
    (m,n)&\mapsto mn \in \bZ\\
    (f:m\to m,g:n\to n)&\mapsto nf+mg \in \bZ/2,
  \end{align*}
  The symmetry isomorphism is given by
  \[
  c_{m,n}^{\otimes}=\begin{cases}
    0 & \text{if } \binom{m}{2} \binom{n}{2} \text{ is even}\\
    \eta_{nm} & \text{if } \binom{m}{2} \binom{n}{2} \text{ is odd}.
  \end{cases}
  \]
  By \cite{May09construction}, $B\bS$ is an $E_{\infty}$ ring space.

  Furthermore, the category $\sE$ described above is a bipermutative
  category, with second product given by multiplication in $\bN$.
  This models the cartesian product of finite sets. The map $\xi:\sE
  \to \bS$ is a bipermutative functor, and $B\xi: B\sE \to B\bS$ is
  therefore an $E_{\infty}$ ring map.
\end{rmk}

\begin{prop}\label{prop:model-trunc-sphere}
  Let $QS^0$ be the zeroth space of the sphere spectrum. Then there is
  a map of $E_{\infty}$ ring spaces
 \[
 \overline{B\xi}:QS^0\longrightarrow B\bS
 \]
 which is the Postnikov 1-truncation of $QS^0$.
\end{prop}
\begin{proof}
  By Remark~\ref {rmk:S-biperm}, we have a map of $E_{\infty}$ ring spaces
  $B\xi : B\sE \to B\bS$. Since $B\bS$ is group-like, this map factors
  through the group completion of $B\sE$, which is equivalent to
  $QS^0$:
  \[
  \xym{
    B\sE \ar[r]^{B\xi} \ar[d] & B\bS.\\
    QS^0 \ar[ru]_{\overline{B\xi}} &
  }
  \]
  The map $\overline{B\xi}$ is an isomorphism on $\pi_0 = \bZ$,
  $\pi_1 = \bZ/2$, and thus it is the Postnikov
  1-truncation.
\end{proof}

Let $(\sC, \oplus, I)$ be any Picard groupoid. By
Theorem~\ref {prop:skel-strictification}, we can assume without loss of
generality that $\sC$ is both skeletal and permutative.  Then each object
$x$ in $\sC$ has a strict inverse, $x^*$, so that $x \oplus x^* = I =
x^* \oplus x$.  There is a natural action of $\bS$ on $\sC$
\[
\bS \times \sC \fto{\cdot} \sC
\]
defined on objects as follows:
\begin{align*}
  0 \times x & \mapsto I\\
  1 \times x & \mapsto x \\
  n \times x & \mapsto ((n-1) \cdot x) \oplus x & \text{ for } n > 1\ \\
  n \times x & \mapsto |n| \cdot x^{\ast} & \text{ for } n < 0.
  \intertext{Let $c$ denote the symmetry of $\sC$.  The action $\cdot$
    on morphisms is defined by:}
  \eta_2 \times 1_x & \mapsto c(x,x), & \\
  \eta_n \times 1_x & \mapsto c(x,x) \oplus 1_{(n-2) \cdot x} &
  \text{ for } n \neq 2.
\end{align*}

\begin{prop}\label{prop:model-eta-action}
  Let $X$ be a stable one-type modeled by a Picard groupoid
  $\sC$, so $B\sC \hty X$.  Then the action of the truncated sphere
  spectrum on $X$ is modeled by the action of $\bS$ on $\sC$.
\end{prop}
\begin{proof}
  The action of $\bS$ on $\sC$ passes to an action of $B\bS$ on $B\sC$
  which is homotopic to that of the group completion $QS^0$:
  The top triangle in the diagram below commutes because $B\sC$ is
  group complete; the bottom commutes because the group completion
  abelianizes $\pi_1$ and hence the action of even permutations (the
  alternating group) is trivial.
  \[\xym@C=3cm{
    B\sE \times B\sC \ar[r] \ar[d] & B\sC\\
    QS^0 \times B\sC\hspace{3pt} \ar[ur] \ar[d] & \\
    B\bS\,\, \times B\sC \ar[ruu] & \\
  }\]
\end{proof}
\begin{rmk}
  An alternate argument for Proposition~\ref {prop:model-eta-action} notes that
  the action of the truncated sphere spectrum on $B\sC$ determines and
  is determined by the unique nontrivial Postnikov invariant
  \[
  K(\pi_0 B\sC, 0) \fto{k_0} K(\pi_1 B\sC, 2)
  \]
  which is given by precomposition with $\eta$.  The discussion preceding
  Corollary~\ref {cor:model-postnikov-invar} shows that this Postnikov invariant
  is modeled by the stable quadratic map $q \cn \pi_0\sC \to \pi_1\sC$
  given by $q(x) = c(x,x)$.  This, in turn, determines and is
  determined by the action of $\bS$ on $\sC$ since $\eta_2$ acts by
  the symmetry $c$.  In Section~\ref {sec:model-Postnikov-invars} we define
  the Postnikov invariant of a Picard groupoid and show that
  it models the Postnikov invariant of $B\sC$
  (Corollary~\ref {cor:Postnikov-invars}).

\end{rmk}


\section{Cokernels of Picard groupoid maps}
\label{sec:cokernels}

Here we describe the cokernel of a map of Picard groupoids and the
resulting exact sequence in homotopy groups.


\begin{defn}[Bigroupoid]
  A \emph{bigroupoid} is a bicategory $\sG$ in which the 1-cells
  are invertible up to 2-isomorphism and the 2-cells are isomorphisms.  The set $\pi_0
  \sG$ is given by the equivalence classes of objects.  For an object
  $x \in \sG$, the group $\pi_1(\sG,x)$ is given by the isomorphism
  classes of 1-endomorphisms of $x$. The group $\pi_2(\sG, x)$ is
  given by the 2-endomorphisms of $1_x$, the identity 1-cell of $x$.
\end{defn}

\begin{defn}[Cokernel {\cite{vitale2002pbe}}]
  Let $F \cn \sC \to \sD$ be a map of Picard groupoids.  The cokernel
  of $F$ is a bigroupoid $\Coker(F)$ defined as follows: The objects
  of $\Coker(F)$ are the objects of $\sD$.  The 1-cells between objects
  $x$ and $y$ are pairs $(f,n)$, where
  \[
  x \fto{f} y \oplus F(n)
  \]
  is a morphism of $\sD$.  The 2-cells between $(f,n)$ and $(f',n')$
  are given by morphisms $\alpha \cn n \to n'$ of $\sC$ such that the
  following diagram commutes:
  \[\xym{
    & x \ar[dl]_-{f} \ar[dr]^-{f'} & \\
    y \oplus F(n) \ar[rr]_-{1 \oplus F(\alpha)} & & y \oplus F(n')
  }\]
  The composite of two 1-cells
  \[
    (f,n) \cn  x \to y \qquad \text{and} \qquad
    (g,m) \cn  y \to z
  \]
  is given by the following composite morphism in $\sD$:
  \[
  x \fto{f} y \oplus F(n) \fto{g \oplus 1} (z \oplus F(m)) \oplus
  F(n) \to z \oplus F(m \oplus n).
  \]
  Further details of the definition can be found in
  \cite[\S2]{vitale2002pbe}; note that the cokernel is denoted
  $\mathrm{Cok\,}(F)$ there.
\end{defn}

There is a natural pseudofunctor $C_F: \sD \to \Coker(F)$ which is the identity
on objects and which takes a morphism $f\cn x \to y$ to the 1-cell
$\widehat{f} = (f,I_\sC)$
determined by the morphism
\[
x \fto{f} y \to y \oplus I_{\sD} \to y \oplus F(I_{\sC}).
\]

\begin{thm}\label{thm:cok-sym-mon}
  The symmetric monoidal structure on $\sD$ induces a symmetric
  monoidal structure on the bicategory $\Coker(F)$. The pseudofunctor
  $C_F$ is symmetric monoidal.
\end{thm}
We prove Theorem~\ref {thm:cok-sym-mon} in Section~\ref {sec:proof-sym-cok}.  In the
remainder of this section we apply this cokernel to model stable
cofibers and Postnikov invariants.

\begin{thm}\label{thm:hty-exact-seq}
  A map of Picard groupoids $F \cn \sC \to \sD$ gives rise to a long
  exact sequence of homotopy groups between $\sC$, $\sD$, and $\Coker(F)$
  \[
  0 \to \pi_2 \Coker(F) \to \pi_1 \sC \to \pi_1 \sD
  \to \pi_1 \Coker(F) \to \pi_0 \sC \to \pi_0 \sD \to \pi_0 \Coker(F) \to 0.
  \]
\end{thm}
\begin{proof}
  Exactness at most positions is verified by \cite{vitale2002pbe},
  noting that the $\Ker(F)$ used there has $\pi_0 \Ker(F) \iso \pi_1
  \Coker(F)$ and $\pi_1 \Ker(F) \iso \pi_2 \Coker(F)$.  Exactness at
  the remaining positions, $\pi_1 \sD$ and at $\pi_1 \Coker(F)$, is
  straightforward from the definitions: An element in $\pi_1 \sC$ is
  represented by a morphism $f:I_{\sC} \to I_{\sC}$.  The image of
  this element in $\pi_1 \sD$ is represented by the composite
  \[
  I_{\sD} \fto{\iso} F(I_{\sC}) \fto{F(f)} F(I_{\sC}) \fto{\iso} I_{\sD}.
  \]
  This composite morphism maps to the trivial element in $\pi_1
  \Coker(F)$ because it factors through a morphism in the image of $F$
  (namely, $F(f)$).  Likewise, if $g:F(X) \to I_{\sD}$ represents an
  element of $\pi_1 \sD$ whose image in $\pi_1 \Coker(F)$ is trivial
  (factors through a morphism in the image of $F$), then the
  trivialization provides an element of $\pi_1 \sC$ whose image in
  $\pi_1 \sD$ is the element represented by $g$.  Exactness at $\pi_1
  \Coker(F)$ is similar, and left to the reader.
\end{proof}

\ifbool{extrarefs}{
The two previous results, together with \cite{Oso10Spectra, GO} show that
the cokernel of Picard groupoids models the cofiber of stable
one-types:
\begin{cor}\label{cor:coker-and-hocofib}
  Let $F:\sC \to \sD$ be a map of Picard groupoids.  Then the
  following is a cofibration sequence of group-like $E_\infty$ spaces:
  \[
  B\sC \to B\sD \to B\Coker(F).
  \]
\end{cor}
\begin{proof}
  Note that since $\Coker(F)$ is a group-like symmetric monoidal
  bicategory, \cite[Theorem 2.1]{Oso10Spectra} and the improved results of \cite{GO} imply that
  $B\Coker(F)$ is a group-like $E_\infty$ space.  Let $C$ be the
  cofiber of the map on classifying spaces.  Then the dashed arrow to
  $B\Coker(F)$ exists by the universal property of $C$, and it is an
  equivalence by Theorem~\ref {thm:hty-exact-seq}.
    \[\xym{
    B\sC \ar[r]
    & B\sD \ar[r] \ar[dr] 
    & C \ar@{-->}[d]^-{\hty} \\
      & 
    & B\Coker(F) 
  }\]
}
{
The two previous results, together with \cite{Oso10Spectra} show that
the cokernel of Picard groupoids models the cofiber of stable
one-types:
\begin{cor}\label{cor:coker-and-hocofib}
  Let $F:\sC \to \sD$ be a map of Picard groupoids.  Then the
  following is a cofibration sequence of group-like $E_\infty$ spaces:
  \[
  B\sC \to B\sD \to B\Coker(F).
  \]
\end{cor}
\begin{proof}
  Note that since $\Coker(F)$ is a group-like symmetric monoidal
  bicategory, \cite[Theorem 2.1]{Oso10Spectra} implies that
  $B\Coker(F)$ is a group-like $E_\infty$ space.  Let $C$ be the
  cofiber of the map on classifying spaces.  Then the dashed arrow to
  $B\Coker(F)$ exists by the universal property of $C$, and it is an
  equivalence by Theorem~\ref {thm:hty-exact-seq}.
    \[\xym{
    B\sC \ar[r]
    & B\sD \ar[r] \ar[dr] 
    & C \ar@{-->}[d]^-{\hty} \\
      & 
    & B\Coker(F) 
  }\]
}
\end{proof}

\subsection{Modeling Postnikov invariants}
\label{sec:model-Postnikov-invars}

\begin{defn}
  Let $\sC$ be a Picard groupoid, and let $\sC_0$ be the category of
  isomorphism classes of objects of $\sC$, with only identity
  morphisms.  Let
  \[
  \al_0 : \sC \to \sC_0
  \]
  be the monoidal functor which takes each object to its isomorphism
  class and takes morphisms to identity morphisms.  Let
  $k_0 = C_{\al_0}$ be the natural pseudofunctor from $\sC_0$ to
  $\Coker(\al_0)$.  We call the sequence
  \[
  \sC \fto{\al_0} \sC_0 \fto{k_0} \Coker(\al_0)
  \]
  the \emph{Postnikov tower} of $\sC$.
\end{defn}

By Theorem~\ref {thm:hty-exact-seq}, $\Coker(\al_0)$ has only one non-trivial
homotopy group, which is $\pi_1 \sC$ in degree two.  We refer to $k_0$
as the Postnikov invariant of $\sC$.  Our terminology is motivated by
the following result.
\begin{samepage}
\begin{cor}\label{cor:Postnikov-invars}
  The Postnikov tower of $\sC$ models the Postnikov tower of $B\sC$.
\end{cor}
\begin{proof}
  This follows immediately from Corollary~\ref {cor:coker-and-hocofib} and the
  fact that $B\sC_0 \hty K(\pi_0, 0)$.
  \[\xym{
    B\sC \ar[r] \ar@{=}[d] 
    & K(\pi_0, 0) \ar[r] \ar[d]^-{\hty} 
    & K(\pi_1, 2) \ar@{-->}[d]^-{\hty} \\
    B\sC \ar[r]^-{B\al_0}
    & B\sC_0 \ar[r]^-{Bk_0} 
    & B\Coker(\al_0) 
  }\]
\end{proof}
  
\end{samepage}



\section{Proof of Theorem \ref{thm:cok-sym-mon}}\label{sec:proof-sym-cok}

\begin{prop}\label{prop:coker-symm-mon}
  The bicategory $\Coker(F)$ is symmetric monoidal.
\end{prop}

To prove this proposition we will construct a double category
$\CCoker(F)$ and use the results of \cite{Shulman2010}, which are
analogous to those of \cite[\S 6]{GG2008lowTricats}.  The idea behind
this method is that it is usually easier to construct symmetric
monoidal double categories than symmetric monoidal bicategories, and
for certain double categories the symmetric monoidal structure lifts
to a symmetric monoidal structure in a related bicategory.

The double category $\CCoker(F)$ is constructed as follows. The category of
objects, $\CCoker(F)_0$, is $\sD$. The category of morphisms,
$\CCoker(F)_1$, has as objects the quadruples $(x,y,f,n)$, where $x$
and $y$ are objects of $\sD$, $n$ is an object of $\sC$ and $f:x\to
y\oplus F(n)$ is a morphism in $\sD$.

A morphism in $\CCoker (F)_1$ from $(x,y,f,n)$ to $(z,v,g,m)$ is
given by a triple $(a, b, \alpha)$, where $a:x\to z$ and $b:y\to v$
are morphisms in $\sD$, and $\alpha : n\to m$ is a morphism in $\sC$,
such that the following diagram commutes
\[
\xymatrix@R=1.4pc@C=1.4pc{
x \ar[r]^-{f} \ar[d]_a & y\oplus F(n) \ar[d]^{b\oplus F(\alpha)}\\
z \ar[r]_-{g} & v\oplus F(m)}
\]
Composition of morphisms in the category $\CCoker(F)_1$ is given by
composition componentwise. We follow the notation from
\cite[Def. 2.1]{Shulman2010} to define the rest of the structure of
the double category.

The unit functor $U$ is defined as
\begin{align*}
 U:\CCoker(F)_0 &\longrightarrow \CCoker(F)_1\\
x &\longmapsto (x,x,\widehat{1}_x, I_{\sC})\\
x\fto{a} z &\longmapsto (a, a, 1 _{I_{\sC}}),
\end{align*}
where $\widehat{1}_x$ is
given by the composition $x\to x\oplus I_{\sD} \to x\oplus F(I_{\sC})$.
The functors for source and target, $S$ and $T$, are given by:
\begin{align*}
 S,T:\CCoker(F)_1 &\longrightarrow \CCoker(F)_0\\
(x,y,f,n) &\longmapsto x, y\\
(a,b,\alpha) &\longmapsto a, b.
\end{align*}
Finally, the composition functor is given by
\begin{align*}
 \odot:\CCoker(F)_1 \times _{\CCoker(F)_0} \CCoker(F)_1 &\longrightarrow \CCoker(F)_1\\
[(x,y,f,n),(y,z,g,m)] &\longmapsto (x,z,f\bullet g, m\oplus n)\\
[(a,b,\alpha),(b,c,\beta)] &\longmapsto (a,c, \beta \oplus \alpha),
\end{align*}
where $f\bullet g$ denotes the composition
\[
 x \fto{f} y\oplus F(z) \fto{g\oplus 1} (z\oplus F(m))\oplus F(n)\to z\oplus F(m\oplus n).
\]
The associativity and unit constraints come from those in the monoidal
structure of $\sC$.

\begin{prop}
  The double category $\CCoker(F)$ is symmetric monoidal.
\end{prop}
\begin{proof}
  The category $\CCoker(F)_0=\sD$ is symmetric monoidal. We now give a
  symmetric monoidal structure to $\CCoker(F)_1$. On objects it is
  given by:
  \[
  (x,y,f,n)\oplus (z,v,g,m)=(x\oplus z,y\oplus v, f\star g, n\oplus m),
  \]
  where $f\star g$ is the composition
  \[
  x\oplus z \fto{f\oplus g} (y\oplus F(n))\oplus (v\oplus F(m))
  \to
  (y\oplus v)\oplus (F(n)\oplus F(m))\to (y\oplus v)\oplus F(n\oplus m).
  \]
  On morphisms it is defined by applying the sum componentwise. The
  associativity, unit, and symmetry constraints are inherited from
  those in $\sC$ and $\sD$.

  The globular isomorphism
  \[\xymatrix{
    ((x,y,f,n)\oplus (y,z,g,m))\odot ((x',y',f',n')\oplus (y',z',g',m')) \ar[d]^{\mathfrak{x}}\\
    ((x,y,f,n)\odot (x',y',f',n'))\oplus ((y,z,g,m)\odot (y',z',g',m'))
  }
  \]
  is given by the structural isomorphism in $\sC$
  \[
  (m\oplus m')\oplus (n\oplus n')\to (m\oplus n)\oplus (m'\oplus n').
  \]
  The globular morphism $\mathfrak{u}:U_{x\oplus y} \to U_x\oplus
  U_y$ is given by the morphism $I\to I\oplus I$ in $\sC$.

  It is clear that all the necessary diagrams commute since they all
  involve compositions of morphisms of the symmetric monoidal
  structures on $\sC$ and $\sD$.
\end{proof}

We recall that a double category is \emph{fibrant} in the terminology
of \cite{Shulman2010} if every vertical 1-morphism has a companion and
a conjoint. These are horizontal 1-morphisms that allow transport of
vertical structure to horizontal structure \cite[\S3]{Shulman2010}:
\begin{defn}
  Let $a:x\to z$ be a morphism in $\sD$. A \emph{companion} for $a$ is given
  by $(x,z,\hat{a},I_{\sC})$, where $\hat{a}$ is the composite
  \[
  x \fto{a} z \to z\oplus I_{\sD} \to z \oplus F(I_{\sC}).
  \]
  The following diagrams commute and therefore the
  equations of \cite[3.1]{Shulman2010} are trivially satisfied.
  \[
  \xymatrix@R=1.4pc@C=1.4pc{
    x \ar[r]^-{\hat{a}} \ar[d]_-a 
    & z\oplus F(I_{\sC}) \ar@{=}[d] \\
    z \ar[r]_-{\widehat{1}_z} & z\oplus F(I_{\sC})}
  \qquad \qquad
  \xymatrix@R=1.4pc@C=1.4pc{
    x \ar[r]^-{\widehat{1}_x} \ar@{=}[d]
    & x \oplus F(I_{\sC}) \ar[d]^-{a}\\
    x \ar[r]^-{\hat{a}}
    & z\oplus F(I_{\sC})}
  \]  
  A \emph{conjoint} for $a$ is given by $(z,x,\check{a},I_{\sC})$, where
  $\check{a}$ is the composite
  \[
  z \fto{a^{-1}} x \to x\oplus I_{\sD} \to x\oplus F(I_{\sD}).
  \]
  This is the companion of $a$ in the double category obtained from
  $\CCoker(F)$ by taking the same category of objects and the opposite
  category of morphisms.
\end{defn}


\begin{proof}[of Proposition~\ref {prop:coker-symm-mon}.]
  The horizontal bicategory $\sH(\CCoker(F))$ is precisely
  $\Coker(F)$.  Since every morphism in $\sD$ has a companion and a
  conjoint, $\CCoker(F)$ is a fibrant double category.  Therefore by
  \cite[Thm 5.1]{Shulman2010} $\Coker(F)$ is symmetric monoidal.
\end{proof}
\begin{rmk}
  Vitale \cite{vitale2002pbe} points out that $\Coker(F)$ is a
  bigroupoid.  We note, moreover, that the objects are weakly invertible
  since they are the objects of $\sD$ with the same monoidal
  structure. Thus $\Coker(F)$ is what one might call a \emph{Picard
    bigroupoid}.
\end{rmk}

\begin{prop}\label{prop:PF-symm-mon}
  The pseudofunctor $C_F:\sD\to \Coker(F)$ is symmetric monoidal.
\end{prop}

\begin{proof}
  We need to specify transformations
  \[
  (\chi_{x,y},\chi _{f,g}): C_F(x)\oplus C_F(y) \to C_F(x\oplus y)
  \]
  and
  \[
  \iota :I\to C_F(I).
  \]

  We let $\chi _{x,y}:x\oplus y \to x\oplus y$ be the identity
  1-cell in $\Coker(F)$, that is, $\widehat{1}_{x\oplus y}$. The
  2-cell
  \[
  \chi _{f,g}: \widehat{f\oplus g} \circ \widehat{1}_{x\oplus
    y}\Rightarrow \widehat{1}_{x'\oplus y'}\circ (\widehat{f}\oplus
  \widehat{g})
  \]
  is given by the unique structural morphism in $\sC$, $I\oplus I \to
  I\oplus (I\oplus I)$. It is an easy verification that this is a
  valid 2-cell in $\Coker(F)$, and that these data forms a
  transformation.

  Similarly, we let $\iota : I\to I$ be the identity 1-cell. The rest
  of the data for a symmetric monoidal pseudofunctor consists of four
  modifications, which are collections of 2-cells. In the four cases,
  the source and target of the modifications have products of copies
  of the unit $I\in \sC$ as their second component, and hence the
  modifications are given by the unique structural morphism connecting
  these two products in $\sC$. The coherence of the symmetric monoidal
  structure on $\sC$ therefore implies that these modifications satisfy all of
  the necessary equations.
\end{proof}



\providecommand{\bysame}{\leavevmode\hbox to3em{\hrulefill}\thinspace}
\providecommand{\MR}{\relax\ifhmode\unskip\space\fi MR }
\providecommand{\MRhref}[2]{%
  \href{http://www.ams.org/mathscinet-getitem?mr=#1}{#2}
}
\providecommand{\href}[2]{#2}

\end{document}